\documentclass[11pt]{article}

\usepackage{amssymb,amsmath}
\usepackage{color}
\usepackage{graphicx}

\newtheorem{Theorem}{Theorem}[section]

\newtheorem{Lemma}{Lemma}[section]
\newtheorem{Definition}{Definition}[section]

\newtheorem{Proposition}{Proposition}[section]

\numberwithin{equation}{section}

\newcommand{\non}{\nonumber}
\newcommand{\cR}{\mathbb R}
\newcommand{\cN}{{\mathbb N}}

\newcommand{\eq}[1]{\mbox{\rm {(\ref{#1})}}}

\begin{document}
\title{\Large\bf Analysis validation of a new continuum model for the evolution of
grain boundaries in polycrystalline materials
}
\author{{\small\sc
Peicheng Zhu$^1$\footnote{E-mail: pczhu@shu.edu.cn}, Xiaoxue Qin$^1$\footnote{Corresponding author. E-mail: qinxiaoxue@shu.edu.cn}}
and Yang Xiang$^2$\footnote{Corresponding author. E-mail: maxiang@ust.hk}
\\
{\small\sc $^1$ Department of   Mathematics, Shanghai University,}\\
{\small\sc  Shanghai 200444, P.R. China} \\
{\small\sc $^2$ Department of Mathematics, Hong Kong University of  Science and  Technology,}\\
{\small\sc  Clear water bay, Kowloon, Hong Kong}
}

\date{}

\maketitle
\noindent{\bf Abstract.}  To  describe the evolution of grain boundaries based
on the underlying microscopic mechanisms of line defects (disconnections) and the integrated effects of a
diverse range of thermodynamic driving forces, Zhang, et al.
in 2017 formulated a continuum equation.
We prove the global-in-time existence, uniqueness, and regularity of the weak solution to an initial-boundary value problem for this model.
The existence, uniqueness of the stationary  solution are also established. Finally, we investigate the large-time behavior of the weak solution of the evolution problem, and show that the solution converges
 to the stationary solution in a suitable sense.
 Numerical simulations are carried out to validate the analysis results.
The main difficulties in the proof of main theorems are due to
a non-local term with singularity,
a non-smooth coefficient of the highest derivative associated with the gradient of the unknown, and the special form
of free energy which is {\it not} uniformly bounded from below.
The key ingredients in the proof are the energy method, an estimate for
a singular integral of the Hilbert type,  Fourier transform of Hilbert transformations, and an estimate
with a weight, for time-derivative of the unknown.


%
%

\section{Introduction}
\label{sec1}

 Grain boundaries (GBs) are  the interfaces between  differently oriented
 crystalline grains, which  are a kind of two-dimensional defects in
 materials. GB migration controls many
 microstructural evolution processes in materials. Since GBs are
 interfaces between crystals, the microscopic
mechanisms by which they move are intrinsically different
from other classes of interfaces, such as solid-liquid interfaces and
 biological cell membranes. A polycrystalline material can be regarded,
 on the mesoscale,  as a network  of grain boundaries. And this GB network
 has a great impact on a wide range of  materials properties, such as
 strength, toughness, electrical conductivity, andthus its evolution is
 important for engineering materials~\cite{SB95}.

Both experiments and atomistic simulations have shown that the
 microscopic mechanism of GB migration is associated
with the motion of topological line defects, i.e.,
disconnections~\cite{A72,HP96,HPL07, KS80,MTP04,RMLC13,RLCMM13,TCHPS17}.
This  dependence on microscopic structures enables broad-range and deep
understanding of GB migration, such as the stress-driven motion and the shear
coupling effect~\cite{CT04, LEWP53}, these cannot be described by the
motion by  classical mean curvature models in which the driving forse  is
capillary forces~\cite{SB95}.

A new continuum model  for motion of grain boundaries based on the
underlying disconnection mechanisms has been developed by Zhang {\it et al.}
\cite{ZHXS17}. This  model combines the effects of a
diverse range of thermodynamic driving forces including the stress-driven
motion and is able to model the shear coupling effect in the process of   GB motions.
This  model has been also generalized to the phenomena with multiple disconnection modes and
GB triple junctions, see~\cite{WTHSX19,WZHSX20,ZHSX21}.

In a previous work~\cite{ZhuYuXiang23}, the authors proved the existence of weak solutions to an
initial-boundary value problem of this  model, under an assumption that
a parameter is suitably large. In the present article, we will remove a restriction  and study
global existence and large-time behavior of solutions to an IBVP of the continuum model, which reads
\begin{eqnarray}
 h_{t} = - M_{d}\big((\sigma_{i}+\tau)b+\Psi H-\gamma H h_{xx}\big)(|h_{x}|+B)
 \label{1.1}
\end{eqnarray}
for $(t,x)\in(0,\infty)\times\Omega$,  where $\Omega=(c,d)$.
The boundary and initial conditions are
\begin{eqnarray}
 h |_{x=c} &=& h |_{x=d},\quad h_{x} |_{x=c} = h_{x} |_{x=d},\ (t,x)\in(0,T_{e})\times\partial\Omega,
 \label{1.2}\\
 h(0,x) &=& h_{0}(x),\ x\in\Omega,
 \label{1.3}
\end{eqnarray}
where
\begin{eqnarray}
 \sigma_{i}(t,x) = {\rm P.V.}\int_{-\infty}^{\infty}\frac{K\beta h_{x}(t,x_{1})}{x-x_{1}}dx_{1},
 \label{1.4}
\end{eqnarray}
and
$$
 \beta=\frac{b}{H},\ K = \frac{\mu}{2\pi(1-\nu)},\
  B = \frac{2H}{a} e^{-F_{d}/(k_{B}T)}.
$$

In Equation \eqref{1.1} $h$ is the unknown function that is the height of grain
boundary from the reference line, and $\sigma_{i}(x,t)$ is the stress due to the
elastic interaction between disconnections based on their dislocation
nature~\cite{HL,Xiang06}.
The parameters $M_{d}$, $b$,  $H$, $\mu$, $\nu$, $\gamma$, $\tau$, $a$,  $k_{B}$, $T$ and $F_{d}$ are,
respectively, the mobility constant, the Burgers vector, step height of a disconnection,    the shear modulus,
Poisson ratio,   the GB energy,   the applied stress, the lattice constant, the
Boltzmann constant, the temperature, and  the disconnection formation energy. And
$\Psi$ is energy jump across the GB, the parameter $B$ is associated with the equilibrium density of the
disconnection pairs, and $\frac{1}{a} e^{-F_{d}/(k_{B}T)}$ is the equilibrium
disconnection density.

Note that the regime of $B\rightarrow 0$ means small equilibrium disconnection
density, and  when $B=0$, the
equation \eqref{1.1} is degenerate at those points where $h_x=0$. Numerical
results in Ref.~\cite{ZHXS17} showed that sharp corners may be developed in the
GB profile in the case of $B=0$.

The  difficulties in the proofs of the existence and uniqueness theorems come from
the non-local term with singularity together with a non-smooth coefficient
associated with $|h_x|$ of the highest derivative $h_{xx}$,  the degeneracy
of the equation in the case of $B=0$, and from that the free energy which contains a
 term of the unknown itself and consequently is {\it not} bounded from below.
Actually the total free energy reads
$$
 M_{d}\int_{0}^{L}\big(\frac\gamma2 H |h_{x}|^2 + (\tau b + \Psi H)h + \frac{b}2\sigma_{i}h\big)dx,
$$
the linear term of $h$ can not be bounded by the other terms in the case of
periodic boundary conditions, hence this free energy is not uniformly bounded from below and
causes difficulty when deriving {\it a priori} estimates.

Our strategies for overcoming these difficulties are as follows. First to
 estimate this singular integral term,
we employ a theorem in the book by Stein~\cite{Stein}, and the Fourier transformation of
 a singular integral of Hilbert type, see, e.g., \cite{King09}. Regularization is performed
so that the coefficient of the $h_{xx}$ term is smooth and uniformly bounded
from below, and to estimate the linear term of unknown we make use of a reciprocal weighted norm for
$h_t$ so that this term is bounded by the  $L^2(\Omega)-$norm of $h_{x}$ multiplied by a
small number, see~Lemma~\ref{lm3.2}.
Then  compactness lemmas are employed to obtain the results for
the original equations.

The dependence  on non-smooth gradient terms in the
coefficient of the highest derivative also appeared in the phase field models
proposed by Alber and Zhu  in~\cite{AZ06,AZ08} to describe the evolution of an
interface driven by configurational forces, and existence of weak solutions have
been studied in~\cite{AZ07,AZ11,BZGZ24,Zhu11B}.
Acharya, et al~\cite{AMZ10} and Hilderbrand, et al~\cite{HM10} studied models with similar
non-smooth gradient terms too.

\vskip0.5cm
\noindent{\bf Non-dimensionalized form of the equation.}  Using $M_d\mu$ as the time unit,
$\mu$ the unit of $\sigma_i$, $\tau$ and $\Psi$, $L_0$ the unit of the length scale
of the continuum equation, and $\mu L_0$ the unit of $\gamma$, we then obtain the dimensionless
form of the equation. Further introducing parameters
$$
 \alpha_{1} = \gamma H,\
 \alpha_{2} =  b, \
 \alpha_{3} =  \tau b+\Psi H,
$$
where all the quantities are in dimensionless form, equation  \eqref{1.1} can be written as
\begin{eqnarray}
 h_{t} - \frac{\alpha_{1}}{2}\left( |h_{x}|h_{x}  + 2Bh_{x}\right)_x
 + (\alpha_{2}\sigma_i + \alpha_{3} )(|h_{x}| + B) =0.
 \label{1.5}
\end{eqnarray}
Here we have used the  formula
 $ (|y|y)^{'} = 2|y|$. From now on, we will use this non-dimensionalized equation with
 the dimensionless parameters described above.

\vskip0.25cm
To define weak solutions to the initial-boundary value
 problem  (\ref{1.1}) -- (\ref{1.3}), we denote by $\Omega=(c,d)$
    a bounded open interval with constants
  $c<d$,  by $T_{e}>0$   an arbitrary constant, and by
   $Q_{T_e}$ the domain $(0,T_{e})\times\Omega$.  Define
$$
 (v_1,v_2)_{Z}=\int _{Z}v_1(y)v_2(y)dy
$$
for $Z=\Omega$ or $Z=Q_{T_{e}}$. Moreover, if $v$ is a function
defined on $Q_{T_{e}}$, we use $v(t)$ to represent the mapping
$x\mapsto v(t,x)$ and sometimes write $ v = v(t)$ for convenience.

\vskip0.5cm
\noindent{\bf Statement of the main results.} Our main results are concerned
with the existence and uniqueness of weak solution to an initial-boundary
 value problem.

\begin{Definition}\label{def1.1}
 Let $h_{0}\in L^{1}(\Omega)$. A function $h$ with
\begin{eqnarray}
 h\in  L^{2}(0,T_{e};H^{1}_{{\rm per}}(\Omega))
 \label{1.6}
\end{eqnarray}
is called a weak solution to  problem \eqref{1.1} -- \eqref{1.3}, if for
all $\varphi\in C_{c,x{\text-}per}^{\infty}([0,T_{e})\times\Omega)$, there holds
\begin{equation}
 (h,\varphi_{t})_{Q_{T_{e}}} - \frac{\alpha_{1}}{2} ( |h_{x}|h_{x} + 2B h_{x},\varphi_{x} )_{Q_{T_{e}}}
 - ((\alpha_{2}\sigma_i + \alpha_{3} )(|h_{x}| + B)  , \varphi) _{Q_{T_{e}}}
 + (h_{0},\varphi(0))_{\Omega}=0.
 \label{1.7}
\end{equation}

\end{Definition}

We then have
\begin{Theorem}\label{thm1.1} Assume that   $h_{0}\in H^{1}_{{\rm per}}(\Omega) $.
Then there exists a unique
  weak solution $h$ to
 problem \eqref{1.1} -- \eqref{1.3} with $B>0$, which in addition
to \eq{1.6}, satisfies
\begin{eqnarray}
 h\in L^{\infty}(0,T_{e};H^{1}_{{\rm per}}(\Omega)),\ \ h_{x}\in L^{2}(0,T_{e};H^{1}_{{\rm per}}(\Omega))
 \cap L^{{3}}(Q_{T_{e}}),
 \label{1.8} \\
 h_{t}\in L^{\frac{3}{2}}(Q_{T_{e}}),\ \
 |h_{x}|h_{x} \in L^{\frac{3}{2}}(0,{T_{e}}; W^{1,\frac{3}{2}}_{{\rm per}}(\Omega)).
 \label{1.9}
\end{eqnarray}

\end{Theorem}

We are also interested in the limit as $B\to 0$.
\begin{Definition}\label{def1.1a}
 Let $h_{0}\in L^{1}(\Omega)$. A function $h$ with
\begin{eqnarray}
 h\in  L^{2}(0,T_{e};H^{1}_{{\rm per}}(\Omega))
 \label{1.10B}
\end{eqnarray}
is called a weak solution to  problem \eq{1.1} -- \eq{1.3} with $B=0$, if for
all $\varphi\in  C_{c,x{\text-}per}^{\infty}([0,T_{e})\times\Omega)$, there holds
\begin{eqnarray}
 (h,\varphi_{t})_{Q_{T_{e}}} - \frac{\alpha_{1}}{2} ( |h_{x}|h_{x} ,\varphi_{x} )_{Q_{T_{e}}}
 - ((\alpha_{2}\sigma_i + \alpha_{3} ) |h_{x}|, \varphi) _{Q_{T_{e}}}
 +  (h_{0},\varphi(0))_{\Omega}=0.
 \label{1.11B}
\end{eqnarray}

\end{Definition}

We denote a solution to  problem \eq{1.1} -- \eq{1.3} by $h_B$, then $h_B$
converges $h$ almost everywhere $(t,x)$ over $Q_{T_e}$, and $h$ satisfies \eq{1.11B}.
\begin{Theorem}\label{thm1.1B} Assume that
$h_{0}\in H^{1}_{{\rm per}}(\Omega) $. Then there exists a weak solution $h$ to
 problem \eqref{1.1} -- \eqref{1.3} with $B=0$, which in addition
to \eqref{1.10B}, satisfies
\begin{eqnarray}
 h\in L^{\infty}(0,T_{e};H^{1}_{{\rm per}}(\Omega)),\ \ h_{x}\in  L^{{3}}(Q_{T_{e}}),
 \label{1.12aB} \\
 h_{t}\in L^{\frac{4}{3}}(Q_{T_{e}}),\ \
 |h_{x}|h_{x} \in L^{\frac{4}{3}}(0,{T_{e}}; W^{1,\frac{4}{3}}_{{\rm per}}(\Omega)),\\
 (|h_{x}|h_{x})_t \in L^{1}(0,{T_{e}}; H^{-2}_{{\rm per}}(\Omega)).
 \label{1.12B}
\end{eqnarray}

\end{Theorem}

We also investigate the existence and asymptotic stability of the stationary solution for the case that
$B$ is positive. First we give
its definition.
\begin{Definition}\label{def1.3} A function $\hat h(x)$ which satisfies
\begin{eqnarray}
  \frac{\alpha_{1}}{2} ( |\hat h_{x}|\hat h_{x} + 2B \hat h_{x},\varphi_{x} )_{\Omega}
  + ((\alpha_{2}\hat \sigma_i + \alpha_{3} )(|\hat h_{x}| + B), \varphi) _{\Omega}
  &=& 0,\  \forall \varphi\in C^\infty_{\rm per}(\Omega), \label{1.13a} \\
  \int_\Omega\hat h(x)dx
  &=& \alpha
  \label{1.13}
\end{eqnarray}
is called a stationary solution to problem  \eqref{1.1} -- \eqref{1.3},
where $\alpha$ is an arbitrarily given constant
 and $\hat \sigma_i = {\rm P.V.}\int_{-\infty}^{\infty}\frac{K\beta \hat h_{x}(x_{1})}{x-x_{1}}dx_{1}$.

\end{Definition}

It is easy to show that such defined weak solution  exists, however its uniqueness needs more restriction of some parameters.
  some estimates on $\hat h$ are derived, and $\hat h(x)$ is, under small perturbation, asymptotically
  stable, which are listed below in Theorem~\ref{thm1.3}.
\begin{Theorem}[Existence]\label{thm1.3} Assume that $B>0$. Then there exists a weak solution $\hat h$ in the sense of
Definition~\ref{def1.3} to problem \eqref{1.1} -- \eqref{1.3}.
Moreover if $\gamma$  is suitably greater than  $ C_1 N$ where $ C_1 $ is a positive constant and
$ N :=\max\{\tau,b,\Psi\}$, then the solution $\hat h$  is unique and the strength of is
small in the sense that there exists a  small number $\delta$ such that
\begin{eqnarray}
   \| \hat h_{xx} \|  +  \| \hat h_{x}\|_{L^\infty(\Omega)} \le  \delta.
 \label{1.14}
\end{eqnarray}

\noindent{\bf (Asymptotical stability)}  For the case that $B>0$,
 assume that there exists a sufficiently small constant $\varepsilon$
   such that $\|h_0-\hat h\|_{H^1(\Omega)}\le \varepsilon$,  and $\gamma> C_1 N$.
Then there exists a unique solution $h$   to problem  \eqref{1.1} -- \eqref{1.3}, such that
$h-\hat h\in L^\infty(0,\infty;H^1_{\rm per}(\Omega))$,  $h_x - \hat h_x\in L^2(0,\infty;H^2_{\rm per}(\Omega))$, and
\begin{eqnarray}
  \sup_{c\le x \le d} |h(t,x) - \hat h (x) - (\bar h - \bar {\hat h}) |    \to 0,\ {\rm as}\  t\to \infty,
 \label{1.15}
\end{eqnarray}
here $\bar  h - \bar {\hat h}  = \frac{1}{|\Omega|}\int_{\Omega}h(t,x) - \hat h (x)dx
= \frac{1}{|\Omega|}\int_{\Omega}h(t,x)dx - \frac{1}{|\Omega|}\alpha$ with the same $\alpha$ as in
\eq{1.13}, and $|\Omega|$ denotes the measure of $ \Omega $.

\end{Theorem}

\vskip0.5cm
\noindent{\bf Notations.}
$C$ denotes a  universal constant  which
may vary from line to line.  and $C(\cdot)$ depends on its argument(s).
Greek letters $\varepsilon,\ \zeta$ are small positive numbers which are normally
 assumed to be small.
${T_e}$ (or $t_e$) denotes a positive constant related to time, the life of a solution.

 Let $p,\ q$ be  real numbers such that $p,\ q\ge 1$.
 Let $\cN$ be the set of natural number and $\cN_+ = \cN\cup \{0\}$, and $\cR^d$ be $d$-dimensional
 Euclidean space.

 $\Omega$ denotes an open, bounded, simple-connected domain in $\cR^d$ with natural number
 $d$, with smooth boundary $\partial \Omega$. It represents
the material points of a solid body. $Q_t=(0,t)\times
\Omega$, and its parabolic boundary ${\cal P}Q_t$ is defined by
$ {\cal P}Q_t:=(\partial\Omega\times [0,t))\cup(\Omega\times\{0\})$.

 $L^p(\Omega)$ are the Sobolev spaces of $p$-integrable
real functions over $\Omega$ endowed with the norm
$$
 \|f\|_{L^p(\Omega)} = \left(\int_\Omega|f(x)|^pdx\right)^\frac1p, \mbox{ if }p<\infty;
 \ \|f\|_{L^\infty(\Omega)} = \mbox{ ess} \sup_{x\in\Omega}|f(x)|.
$$
 Throughout this article, the norm of $L^2(\Omega)$ is denoted by $\|\cdot\|$,
 and the norm of $L^2(Q_{T_{e}} )$ is denoted by $\| \cdot \|_{Q_{T_{e}}}$.

Let $\Omega$ be an $n$-dimensional cuboid. Let $\alpha\in \cN_0^d$  be a multi-index
and $|\alpha|$ be its length, where  $\cN_0=\cN\cup\{0\}$.
$D^\alpha f$ is the $|\alpha|$-th order weak derivatives. Define the space
 $W^{m,p}_{{\rm per}}(\Omega)=\{ f\in L^p(\Omega)\mid D^\alpha f\in L^p(\Omega) \mbox{ for all } \alpha
\mbox{ such that }
|\alpha|\le m,  \mbox{ and } \gamma_jf|_{{\rm on\ one\ face}} = (-1)^j \gamma_jf|_{{\rm on\ the\ corresponding\ face}},
\mbox{ for } j=0,1,\cdots,m-1\}$ endowed with norm
$$
 \|f\|_{W^{m,p}_{{\rm per}}} = \left( \sum_{|\alpha|\le m} \|D^\alpha f\|_{L^p(\Omega)}^p \right)^\frac1p,
$$
where $\gamma_j$ are the trace operators.

 $C_{c,x{\text-}per}^{\infty}([0,T_{e})\times\Omega)$ denotes the space of all
functions with compact support which have infinitely many order derivatives and satisfy space-periodic boundary conditions.
 And $W_0^{m,p}(\Omega)$ is the closure of $C_0^\infty(\Omega)$ in   the norm $\|\cdot\|_{W^{m,p}(\Omega)}$.
For $p=2$, $H^m_{{\rm per}} (\Omega):= W^{m,2}_{{\rm per}}(\Omega)$,
 $H^{-m}_{{\rm per}} (\Omega)$ denotes the dual space of $H^m_{{\rm per}} (\Omega)$.
$H^m _0(\Omega):= W^{m,2}_0(\Omega)$.

Let $q,p\in \cR$ such that $q,p\ge 1$.
\begin{eqnarray}
 L^q(0,T;L^{p}(\Omega))
  &:=&\Big\{f \mid f \mbox{ is Lebesgue measurable such that }  \non\\
  && \|f\|_{L^q(0,t;L^p(\Omega))} := \left(\int_0^t\left(\int_\Omega
  |f|^pdx\right)^{\frac{q}p}d\tau \right)^\frac1q<\infty \Big\},
 \non
\end{eqnarray}
and
$L^q(0,t;W^{m,p}_{{\rm per}}(\Omega)):=\{f\in L^q(0,t; L^{p}(\Omega))
\mid \int_0^t \|f(\cdot,\tau)\|_{W^{m,p}_{{\rm per}}} ^q d\tau<\infty \}$. See, e.g., \cite{Lions}.
 We also need some function spaces: For non-negative integers  ${m},  {n}$,
real number $\alpha\in(0,1)$ we denote by $C^{m+\alpha}(\overline{\Omega})$ the
space of \textit{m}-times differentiable functions on $\overline{\Omega}$, whose
$m$th derivative is H\"{o}lder continuous with exponent $\alpha$. The space
$C^{\alpha,\frac{\alpha}{2}}(\overline{Q}_{T_{e}})$ consists of all functions
on $\overline{Q}_{T_{e}}$, which are H\"{o}lder continuous in the parabolic distance
$$
 d((t,x),(s,y)):=\sqrt{|t-s|+|x-y|^{2}}
$$
$C^{m,n}(\overline{Q}_{T_{e}})$ and $C^{m+\alpha,n+\frac{\alpha}{2}}(\overline{Q}_{T_{e}})$,
respectively, are the spaces of functions, whose ${x}$-derivatives up to
order ${m}$ and ${t}$-derivatives up to order \textit{n} belong to
$C(\overline{Q}_{T_{e}})$ or to $C^{\alpha,\frac{\alpha}{2}}(\overline{Q}_{T_{e}})$,
respectively.

\vskip0.2cm
\noindent{\bf Organization of rest of this article.} The main results of this article are
Theorem~\ref{thm1.1}  and Theorem~\ref{thm1.3}.
 The remaining sections  are devoted to the proofs of these theorems.
In Section~2, we prove by employing the method of continuation of local solutions. To this end
we first construct an approximate initial-boundary value problem and prove local-in-time
 existence of weak solutions in suitable spaces to this problem.
and then we derive in Section~3 {\it a priori} estimates which are uniform in time $T$,
 in which we use $L^p$-estimate, Fourier transform of singular integrals, reciprocal weighted
 estimate of $h_t$ from which we find for $\int_\Omega h dx$ a bound in terms of $\varepsilon \| h_x\|_{L^2( \Omega)} $
  where $\varepsilon$ is a positive constant which can chosen arbitrarily small.
 In Section~4,  we prove the existence and uniqueness of stationary solution and show that the solution to the original
  initial-boundary  value problem  converges asympotically to this steady solution. Section~5 is
devoted to the comparison  of the asymptotic state of weak solutions with numerical solutions.

%

\section{Construction of local solutions}
\label{sec2}

To prove Theorem~\ref{thm1.1}, we employ the method of continuation of
local solutions and as first step we investigate the existence of local weak
solutions in
$$
 W:=\{h\mid h\in L^{\infty}(0,T_{e};H^{1}_{{\rm per}}(\Omega)),\  h_{x}\in L^{2}(0,T_{e};H^{1}_{{\rm per}}(\Omega))
 \cap L^{{3}}(Q_{T_{e}})\},
$$
 to problem \eqref{1.1} -- \eqref{1.3}.
 For $\hat h\in W_M:=\{h\in W \mid  \|h\|_{H^1(\Omega)}\le M\}$
 where $M\ge \sqrt{2} \|h_0\|_{H^1(\Omega)}$,
 we consider the following approximate initial-boundary value  problem:
\begin{eqnarray}
   h_{t} -   \alpha_{1}}  ( |{h_{x}|  +  B)  h_{xx}
 &= & \hat {\cal F},\ {\rm in}\ Q_{T_e},\label{2.1}\\
 \hat {\cal F} &:=& - (\alpha_{2}\hat\sigma_i  + \alpha_{3} )(|\hat h_{x}|  + B),
 \label{2.1a}\\
 h |_{x=c} &=& h |_{x=d},\quad h_{x} |_{x=c} = h_{x} |_{x=d}, \ {\rm on}\  \partial\Omega\times [0,T_{e}],\ \
 \label{2.2}\\
 h(0,x)  &=& h_{0}(x), \ {\rm in}\  \Omega.
 \label{2.3}
\end{eqnarray}

It is easy to show that the above problem has a solution $h\in W$, we thus  define an
operator $\hat h\mapsto {\cal T}\hat h =: h$, ${\cal T}: W_M\to W_M$, and by iteration we construct a
sequence of solutions $h^n$ with $n=1,2,3\cdots$.
We are going to derive {\it a priori} estimates for these approximate solution sequence,
 and conclude   the existence of local weak solutions. We have
\begin{Theorem}\label{thm1.1A1} Assume that
$h_{0}\in H^{1}_{{\rm per}}(\Omega)$. Then there exists a positive constant $t_0$, such that
 weak solution $h$ over $Q_{t_0}$ to
 problem \eqref{1.1} -- \eqref{1.3}, which in addition
to \eqref{1.10B}, satisfies
\begin{eqnarray}
 h\in L^{\infty}(0,t_{0};H^{1}_{{\rm per}}(\Omega)),\ \ h_{x}\in  L^{{3}}(Q_{t_0}),
 \label{4.4a} \\
 h_{t}\in L^{\frac{3}{2}}(Q_{t_0}),\ \
 |h_{x}|h_{x} \in L^{\frac{3}{2}}(0,{t_0}; W^{1,\frac{3}{2}}_{{\rm per}}(\Omega)),\label{4.4b}\\
 (h_{x})_t \in L^{\frac{3}{2}}(0,{t_0}; W^{-1,\frac{3}{2}}_{{\rm per}}(\Omega)).
 \label{4.4c}
\end{eqnarray}

\end{Theorem}

Assume that $\|\hat h\|_{H^1(\Omega)}^2\le M^2 := 2\|h_0\|^2_{H^1(\Omega)} $. Since the {\it a priori} estimates
for local solutions are easier to establish than those for global solutions, we state them
briefly.

\vskip0.2cm
\noindent{\it Step 1.} Multiplying \eq{2.1} by $h$ and integrating the resulting equation with respect to
$x$ over $\Omega$, then using integration by parts, we obtain
\begin{eqnarray}
  \frac12\frac{d}{dt} \|h \|^2 +  \alpha_{1}  \int_\Omega \left( \frac12 | h_{x}|^3 +  B  | h_{x}|^2\right)dx
 & = & \int_\Omega \hat {\cal F}hdx,\non\\
 & \le & C\int_\Omega (|\hat\sigma_i |+1))(|\hat h_x|+1)|h|dx \non\\
 & \le & C(\|\hat h_x\|_{L^\infty(\Omega)}+1)(\|\hat\sigma_i \|+1)) \|h\| \non\\
 & \le & C(\|\hat h_x\|_{H^1(\Omega)}+1)(\|\hat h_x\|+1)) \|h\| .
 \label{2.5}
\end{eqnarray}
Here the H\"older inequality and Sobolev's embedding theorem were applied, and we used
the estimate on singular integrals, e.g., Stein~\cite{Stein}  to deal with the
 non-local term $\hat\sigma_i$.
Hence
\begin{eqnarray}
   \frac{d}{dt} \|h \|  \le  C(\|\hat h_x\|_{H^1(\Omega)}+1)(\|\hat h_x\|+1))  ,
 \label{2.6}
\end{eqnarray}
integrating with respect to $t$ we arrive at
\begin{eqnarray}
 \|h (t)\|  & \le &  C\int_0^t(\|\hat h_x\|_{H^1(\Omega)}+1)(\|\hat h_x\|+1))d\tau + \|h (0)\|\non\\
 & \le & C (\sup_{0\le \tau\le t}\|\hat h_x(\tau)\|+1)
 \left(\int_0^t(\|\hat h_x\|^2_{H^1(\Omega)}+1)d\tau \right)^\frac12  + \|h (0)\|  \non\\
 & \le & C\varepsilon(\sup_{0\le \tau\le t}\|\hat h_x(\tau)\|+1)   + \|h (0)\| ,
 \label{2.7}
\end{eqnarray}
where the absolute continuity of a Lebsgue integral  was used.

\vskip0.2cm
\noindent{\it  Step 2.} Multiplying \eq{2.1} by $-h_{xx}$ and integrating the resulting equation with respect to
$x$ over $\Omega$,   we arrive at
\begin{eqnarray}
 \frac12\frac{d}{dt} \|h_x \|^2 +  \alpha_{1}  \int_\Omega \left(  | h_{x}|  +  B\right)  | h_{xx}|^2dx
 & = & - \int_\Omega \hat{\cal F}h_{xx}dx \non\\
 & \le & C(\|\hat h_x\|_{L^\infty(\Omega)}+1)(\|\hat\sigma_i \|+1)) \|h_{xx}\| \non\\
 & \le & C(\|\hat h_x\|_{H^1(\Omega)}+1)(\|\hat h_x\|+1)) \|h_{xx}\| \non\\
 & \le & C_\varepsilon(\|\hat h_x\|_{H^1(\Omega)}+1)^2(\|\hat h_x\|+1))^2 + \varepsilon \|h_{xx}\| ^2, \non\\
 \label{2.8}
\end{eqnarray}
from which, by choosing $\varepsilon=\frac{\alpha_{1}B}2$, it follows that
\begin{eqnarray}
 &&\frac12 \|h_x \|^2 +  \alpha_{1}  \int_\Omega \left(  | h_{x}|  +  \frac{B}2\right)  | h_{xx}|^2dx   \non\\
 & \le & C\sup_{0\le\tau\le t}(\|\hat h_x(\tau)\|+1))^2\int_0^t(\|\hat h_x\|_{H^1(\Omega)}+1)^2d\tau
 + \frac12 \|h_{0x} \|^2 \non\\
 & \le & C\eta(\sup_{0\le\tau\le t}(\|\hat h_x(\tau)\|+1))^2
 + \frac12 \|h_{0x} \|^2.
 \label{2.9}
\end{eqnarray}
Further applying the interpolation technique, from the equation we infer that
\begin{eqnarray}
  \|h_t \|^\frac32_{L^\frac32(Q_t)}
  \le   C .
 \label{2.10}
\end{eqnarray}

Collecting the above estimates we then have
\begin{Lemma}[A priori estimates] \label{lm2.1}
Suppose that the initial data $h_{0}\in H_{per}^{1}( \Omega )$.
Then there exists a solution $h$ over $Q_{t_0}$ to problem \eq{2.1} --   \eqref{2.3},
 such that for any $t\in [0,t_0]$,
\begin{eqnarray}
 \|h\|^2_{ L^\infty(0,t;H^1(\Omega))} + \|h_t\|^\frac32_{L^\frac32(Q_t)}
  + \int_0^t \left(\|h(\tau)\|^2_{H^2(\Omega)}
 + \|h(\tau)\|^3_{L^3(\Omega)}\right)d\tau \le C_{t_0}.
 \label{2.11}
\end{eqnarray}

\end{Lemma}

\noindent{\bf Proof} of Theorem~\ref{thm1.1A1}. we make use of the following lemmas.
\begin{Lemma}[Aubin-Lions] \label{lm2.3}
Let $B_{0}$ and $B_{2}$ be reflexive Banach spaces and let $B_{1}$
be a Banach space such that $B_{0}$ is compactly embedded in $B_{1}$
and that $B_{1}$ is embedded in $B_{2}$. For $1\le  p_{0}, p_{1}\le +\infty$,
define
$$
 W=\left\{f\mid f\in L^{p_{0}}(0,T;B_{0}),\  \frac{df}{dt}\in L^{p_{1}}(0,T;B_{2})\right\}.
$$
$($\textit{i}$)$ if $p_{0}<+\infty$, then the embedding of $W$ into
$L^{p_{0}}(0,T;B_{1})$ is compact.\\
$($\textit{ii}$)$  if $p_{0}=+\infty$ and $p_{1}>1$, then
the embedding of $W$ into $C([0,T];B_{1})$ is compact.

\end{Lemma}

The proof of this Theorem, we refer to, e.g., \cite{Lions,Roubicek}.

\begin{Lemma} [Weak convergence]\label{WeakConvergence}
Let $1<q<\infty$ and let $(0,T_e)\times \Omega$ be an open
domain in $\cR^+\times \cR^n$.
Assume that $g_n, g \in L^q((0,T_e)\times \Omega )$ satisfy
$$
 \|g_n\|_{L^q((0,T_e)\times \Omega )}\le C, \ \ g_n\to g \ a.\
 e. \ in\ (0,T_e)\times \Omega .
$$
Then $g_n$ converges weakly  to $g$ in $L^q((0,T_e)\times \Omega )$.
\end{Lemma}

Now we recall the definition of the mapping ${\cal T}$ to construct a sequence of
approximate solutions $ h^n $ ($n=1,2,3,\cdots$) which satisfies
\begin{Lemma}\label{lm2.4} There exists a subsequence of $h_{x}^{n}$
(we still denote it by $h_{x}^{n}$) such that
\begin{eqnarray}
 h_{x}^{n} &\to& h_{x} {\rm \ strongly\ in\ } L^{2}(Q_{T_{e}}),
 \label{2.12}\\
  h_{x}^{n}
 &\to&  h_{x}  {\rm \ a.e.\ in\ }   Q_{T_{e}} ,
 \label{2.13}\\
 ~~|h_{x}^{n}| h_{x}^{n}
 &\to& |h_{x}|h_{x} {\rm \ strongly\ in\ }  L^{1}(Q_{T_{e}})
 \label{2.14}\\
 \sigma_i ^n  | h^n_{x}|
 &\rightharpoonup & \sigma_i |h_{x}|  {\rm \ weakly\ in\ }   L^{1}(Q_{T_{e}})
 \label{2.15}
\end{eqnarray}
as $n\to \infty$. Here $\sigma_i ^n = {\rm P.V.}\int_{-\infty}^{\infty}
 \frac{K\beta   h^n_{x}(x_{1})}{x-x_{1}}dx_{1}$.

\end{Lemma}

\noindent\textbf{Proof.}  Let $p_{0}=2, p_{1}=\frac{3}2$ and
$$
 B_{0}=H^{1}(\Omega), \ B_{1}=L^{2}(\Omega), \ B_{2}=W^{-1,\frac{3}2}(\Omega).
$$
These spaces satisfy the assumptions of  Lemma~\ref{lm2.3}. Since
estimate  \eqref{2.11}  implies that
$ h_{xx}^{n}\in L^{2}(0,T_{e};L^{2}(\Omega))$, then
\begin{equation}
 h_{x}^{n}\in L^{2}(0,T_{e};H^{1}(\Omega)),
 \label{2.16}
\end{equation}
and
\begin{equation}
 h_{xt}^{n}\in L^{\frac{3}{2}}(0,T_{e};W^{-1,\frac{3}{2}}).
 \label{2.16a}
\end{equation}
Hence $h_{x}^{n}$ is compact in $L^{2}(Q_{T_{e}}$, and there exists a subsequence of it (we still denote
it by the original notation), such that
$$
 h_{x}^{n}\to h_{x}\ {\rm  strongly\ in}\ L^{2}(Q_{T_{e}}),
$$
 as $n\to 0$. From this it follows that
$$
 h_{x}^{n}\to h_{x}\ {\rm  a.e.\ in}\  Q_{T_{e}} .
$$
 This completes the proof of Lemma~\ref{lm2.4}.

\vskip0.2cm
\noindent{\bf Proof of Theorem~\ref{thm1.1A1}.}
By the results in Lemma~\ref{lm2.4}, we need to study the limit of
 \eq{2.15} to conclude the assertions of Theorem~\ref{thm1.1A1}. First we estimate
\begin{eqnarray}
  \|\sigma_i^n   \|\le C  ,
 \label{2.17}\\
 \|\sigma_i^n   |h_x^n|\| \le \|\sigma_i^n   \|\, \|h_x^n \|\le C.
 \label{2.18}
\end{eqnarray}
Thus there exists a subsequence of $\sigma_i^n$, converges weakly star to $\sigma_i $
in $L^\infty(0,t_0;L^2(\Omega))$. Recalling that $h_x^n\to h_x $ strongly in
$L^\infty(0,t_0;L^2(\Omega))$, one obtains
$
  \sigma_{i }^n |h^n_{x}|
  $  converges weakly-star to $\sigma_i|h _{x}| $
in $L^\infty(0,t_0;L^2(\Omega))$.
  Thus the proof  of Theorem~\ref{thm1.1A1} is complete.

\section{A priori estimates}
 \label{sec3}

In this section we are going to derive {\it a priori} estimates  for solutions to
the  problem \eqref{1.1} -- \eqref{1.3}, which may depend only on $T_e$. For more
convenient, we rewrite the equation in the form:
\begin{eqnarray}
   h_{t} -   \alpha_{1}}  ( |{h_{x}|  +  B)  h_{xx} + (\alpha_{2} \sigma_i  + \alpha_{3} )(|  h_{x}|  + B)
 = 0,\ {\rm a.e.\ in}\ Q_{T_e} \label{3.1}
\end{eqnarray}
with the boundary and initial conditions \eq{2.2} and \eq{2.3}.

\begin{Lemma} \label{lm3.1} For any $t\in[0,T_{e}]$ there hold {\rm (}i{\rm )}
if   $\int_\Omega h(t,x)dx\ge 0$,
\begin{eqnarray}
 \frac{\alpha_1}2\|h_x(t)\|^{2}  + \alpha_3\int_\Omega h(t,x)dx
 + \int_{0}^{t}\int_{\Omega}\frac{|h_t|^2}{|h_{x}|  + B }  dxd\tau
 &\le& C + C \|h_0\|^{2}_{H^1(\Omega)},
 \label{3.2a}
 \end{eqnarray}
 and {\rm (}ii{\rm )} if   $\int_\Omega h(t,x)dx < 0$,
\begin{eqnarray}
  \frac{\alpha_1}2\|h_x(t)\|^{2}
 + \int_{0}^{t}\int_{\Omega}\frac{|h_t|^2}{|h_{x}|  + B }  dxd\tau
  \le  C + C \|h_0\|^{2}_{H^1(\Omega)} + \alpha_3\int_\Omega -h(t,x)dx .
 \label{3.2b}
\end{eqnarray}

\end{Lemma}

\noindent{\bf Remark.} The third term in \eq{3.2a}, $ |h_{x}|  + B$ is regarded as
a weight and thus we can that estimate is reciprocally weighted.

\vskip0.2cm
\noindent\textbf{Proof.}  Multiplying Eq. \eqref{3.1} by $h_t/(|h_x|+B) $,
making use of integration by parts, and invoking the
boundary condition \eqref{2.2}, we arrive at
\begin{eqnarray}
  \frac{d}{dt} \left(\frac{\alpha_1}{2}\|h_x(t)\|^{2} + \alpha_{3}\int_{\Omega} h(t,x)dx\right)
  + \alpha_{2}\int_{\Omega} \sigma_i h_t (t,x)dx
 + \int_{\Omega} \frac{|h_t|^2}{ |  h_{x}| + B }dx=0.
 \label{3.3}
\end{eqnarray}
We need to deal with the term of $\int_{\Omega} \sigma_i h_t (t,x)dx$, and write
$$
 h(t,x) = \sum_{k=-\infty}^{\infty} a_k {\rm e}^{i \frac{2\pi}{L}kx},
$$
where $i$ is the imaginary unit, and where $a_k = a_k(t) := \frac1L \int_{a}^{d} h(t,x )
{\rm e}^{i \frac{2\pi}{L}kx} dx$.
There holds
\begin{eqnarray}
 h_x(t,x) &=& \sum_{k=-\infty}^{\infty} b_k {\rm e}^{i \frac{2\pi}{L}kx},\ {\rm where\ }
 b_k=i k \frac{2\pi}{L},\label{3.4a}\\
 h_t(t,x) &=& \sum_{k=-\infty}^{\infty} c_k {\rm e}^{i \frac{2\pi}{L}kx},\ {\rm where\ }
 c_k = (a_k)':=\frac{d}{dt}a_k.
 \label{3.4b}
\end{eqnarray}
Then the Hilbert transform ${\cal H}(h)$ (see e.g. \cite{King09}) of $h(t,x)$ is
\begin{eqnarray}
 {\cal H}(h) &=& i\sum_{k=1}^{\infty}\left( a_{-k} {\rm e}^{ - i \frac{2\pi}{L}kx}
 -a_k {\rm e}^{i \frac{2\pi}{L}kx}\right),
 \label{3.5}
\end{eqnarray}
and  $ {\cal H}(h_x) $ is
\begin{eqnarray}
 {\cal H}(h_x)
 &=&   i\sum_{k=1}^{\infty}\left( b_{-k} {\rm e}^{ - i \frac{2\pi}{L}kx}
   - b_k {\rm e}^{i \frac{2\pi}{L}kx}\right)
   =   i\sum_{k=1}^{\infty}\left(- i k \frac{2\pi}{L}{\rm e}^{ - i \frac{2\pi}{L}kx}
   -i k \frac{2\pi}{L}   {\rm e}^{i \frac{2\pi}{L}kx}\right)\non\\
 &=&   \frac{2\pi}{L} \sum_{k=1}^{\infty}k\left(  a_{-k}  {\rm e}^{ - i \frac{2\pi}{L}kx}
 + a_{k}   {\rm e}^{i \frac{2\pi}{L}kx}\right).
 \label{3.6}
\end{eqnarray}
Therefore applying the properties of the Fourier transform of the Hilbert transformation, we have
\begin{eqnarray}\int_{\Omega} \sigma_i h_t dx
 &=&  \int_a^{a+L}  \frac{2\pi}{L} \sum_{k=1}^{\infty}k\left(  a_{-k}
 {\rm e}^{ - i \frac{2\pi}{L}kx} + a_{k}   {\rm e}^{i \frac{2\pi}{L}kx}\right)
 \sum_{k=-\infty}^{\infty} c_k {\rm e}^{i \frac{2\pi}{L}kx}  dx
 \non\\
  &=&  \frac{2\pi}{L} \sum_{k=1}^{\infty}k\left(   a_{-k} (a_{-k} )'
 +   a_{k} (a_{k} )' \right) \non\\
  &=&  \frac{ \pi}{L} \frac{d}{dt}\sum_{k=1}^{\infty}k\left(   |a_{-k}|^2  +  |a_{k}|^2 \right),
 \label{3.7}
\end{eqnarray}
which, combined with \eq{3.3}, yields
\begin{eqnarray}
  \frac{d}{dt} \left(\frac{\alpha_1}{2}\|h_x \|^{2} +  \frac{ \pi\alpha_{2}}{L}
  \sum_{k=1}^{\infty}k\left(   |a_{-k}|^2  +  |a_{k}|^2 \right)
  + \alpha_{3}\int_{\Omega} h dx\right)
  + \int_{\Omega} \frac{|h_t|^2}{ |  h_{x}|   + B }dx=0.
 \label{3.8}
\end{eqnarray}
Integrating \eq{3.7} with respect to $t$ over $[0,t]$ one obtains
\begin{eqnarray}
  &&  \frac{\alpha_1}{2}\|h_x \|^{2}
  +  \frac{ \pi\alpha_{2}}{L}  \sum_{k=1}^{\infty}k\left(   |a_{-k}|^2
  +  |a_{k}|^2 \right) + \alpha_{3}\int_{\Omega} h dx
 + \int_0^t\int_{\Omega} \frac{|h_t|^2}{ |  h_{x}|   + B }dxd\tau\non\\
 &\le& C+C\|h_0\|^2_{H^1(\Omega)} .
 \label{3.9}
\end{eqnarray}
Noting that $\sum_{k=1}^{\infty}k\left(   |a_{-k}|^2  +  |a_{k}|^2 \right) $ is
 obviously non-negative, we then arrive at \eq{2.1a} in the case that
$\int_{\Omega} h(t,x)dx\ge 0$, otherwise we move this term to the right-hand
side, and obtain  \eq{3.2b}.
Thus the proof of this lemma is complete.

To infer {\it a priori} estimate from \eq{3.2b} we need to employ the
estimate for the term of $h_t$.
\begin{Lemma} \label{lm3.2} For any $t\in[0,T_{e}]$ which satisfies
 $\int_\Omega h(t,x)dx < 0$, there holds that
\begin{eqnarray}
 \int_{\Omega} -h(t,x)dx  \le   \varepsilon \sup\limits_{0\le\tau \le t}\|h_x(\tau)\|^{2}
  + C_\varepsilon.
 \label{3.10}
\end{eqnarray}
Here $\varepsilon$ is a positive constant which can be chosen sufficiently small.

\end{Lemma}

%
\noindent\textbf{Proof.} We write
\begin{eqnarray}
 0 < \int_{\Omega} -h(t,x)dx  &=& \int_{\Omega} \left(\int_{0}^{t}-h_t(\tau,x)d\tau
 + h_0(x)\right)dx \non\\
  &=& \int_{0}^{t} \int_{\Omega} \frac{-h_t(\tau,x)}{(|h_{x}|  + B)^\frac12}
  (|h_{x}|  + B)^\frac12dx d\tau + \int_{\Omega} h_0 dx \non\\
    &\le & \left(\int_{0}^{t} \int_{\Omega} \frac{|h_t|^2)}{ |h_{x}|  + B }dx d\tau\right)^\frac12
   \left(\int_{0}^{t} \int_{\Omega} ( |h_{x}|  + B) dx d\tau \right)^\frac12
   + \int_{\Omega} h_0 dx .\non\\
    &  &
  \label{3.11}
\end{eqnarray}
Thus invoking \eq{3.2b} we get
\begin{eqnarray}
    \int_{\Omega} -h(t,x)dx
  &\le & C\left( \int_{\Omega} -h(t,x)dx +1 \right)^\frac12
  \left(T_e\sup\limits_{0\le\tau \le t} \|h_{x}(\tau)\|  + 1\right) ^\frac12 + C ,
  \label{3.12}
\end{eqnarray}
from which it follows that
\begin{eqnarray}
 \left(\int_{\Omega} -h(t,x)dx  \right)^\frac12
  &\le &   C \left(T_e \sup\limits_{0\le\tau \le t}\|h_x(\tau)\| + 1\right)   ^\frac12 + C\non\\
    &\le &  \left(\varepsilon  \sup\limits_{0\le\tau \le t}\|h_x(\tau)\|
    + C_\varepsilon\right)  ^\frac12 + C.
  \label{3.13}
\end{eqnarray}
The proof of this lemma is thus complete.

\begin{Lemma} \label{lm3.3} For any $t\in[0,T_{e}]$  there holds that
\begin{eqnarray}
  \|h_x(t)\|^{2} + \int_{\Omega} |h | dx + \int_{0}^{t}\int_{\Omega}
  \frac{|h_t|^2}{|h_{x}|  + B }  dxd\tau
  &\le&  C, \label{3.14a} \\
  \int_{0}^{t}\int_{\Omega} \left( \alpha_1 h_{xx}-(\alpha_2 \sigma_i + \alpha_3)\right)^2(|h_{x}|  + B ) dxd\tau
  &\le&  C, \label{3.14b} \\
  \int_{0}^{t}\int_{\Omega} |  h_{xx}|^2(|h_{x}|  + B ) dxd\tau
  &\le&  C, \label{3.14c}\\
  \int_{0}^{t}\left(\| h_{x}\|^3_{L^\infty(\Omega)} + \| h_{xx}\|^2  \right) d\tau
  &\le&  C,\label{3.14d}\\
  \int_{0}^{t}\left\| | h_{x} |^\frac12  h_{xx} \right\|_{L^\frac32(\Omega)}^\frac32   d\tau
  &\le&  C.\label{3.14e}
\end{eqnarray}

\end{Lemma}

\noindent\textbf{Proof.} If Case (i) in Lemma~\ref{lm3.1} happens, then \eq{3.14a} is done.
Otherwise we make use of
\eq{3.2b} and \eq{3.10} to obtain
\begin{eqnarray}
   \|h_x(t)\|^{2}
 + \int_{0}^{t}\int_{\Omega}\frac{|h_t|^2}{|h_{x}|  + B }  dxd\tau
  \le  C + C \|h_0\|^{2}_{H^1(\Omega)} + \varepsilon \sup\limits_{0\le\tau \le t}
  \|h_x(\tau)\|^{2}. \label{3.15}
\end{eqnarray}
Choosing $\varepsilon$ sufficiently small, we then arrive at \eq{3.14a}.

Estimate \eq{3.14b} is derived from the equation and \eq{3.14a}. With the help of the
simple inequality $(a+b)^2\le 2(a^2+b^2)$,
we write
\begin{eqnarray}
 &&\int_{0}^{t}\int_{\Omega} | \alpha_1 h_{xx}|^2(|h_{x}|  + B ) dxd\tau \non\\
 &\le& 2\int_{0}^{t}\int_{\Omega} \left( \alpha_1 h_{xx}-(\alpha_2 \sigma_i + \alpha_3)\right)^2(|h_{x}|  + B ) dxd\tau
  + 2\int_{0}^{t}\int_{\Omega} \left( \alpha_2 \sigma_i + \alpha_3 \right)^2(|h_{x}|  + B ) dxd\tau\non\\
   &\le&  C + C\int_{0}^{t}\left(\|h_{x}\|_{L^4(\Omega)}^2  + 1 \right)\|h_{x}\|  d\tau\non\\
   &\le&  C + C\int_{0}^{t}\left(\|h_{x}\|^2 _{H^1(\Omega)} + 1 \right)   d\tau
   \le  C.
 \label{3.16}
\end{eqnarray}

Assume that $p$ is a positive number to be determined. We invoke the
H\"older inequality to estimate
\begin{eqnarray}
   \int_{0}^{t} \int_{\Omega}   |\, | h_{x} |     h_{xx}  |^p dx   d\tau
  & = &   \int_{0}^{t}  \| h_{x} \|_{L^\infty(\Omega)}^\frac{p}{2}\int_{\Omega}|  (| h_{x} |^\frac{1}{2}
   |  h_{xx}  |)^pdx  d\tau  \non\\
 & \le &   \left(\int_{0}^{t}  \| h_{x} \|_{L^\infty(\Omega)}^\frac{pq'}{2}d\tau\right)^\frac1{q'}
   \left(\int_{0}^{t}\int_{\Omega}|  (| h_{x} |^\frac{1}{2}   |  h_{xx}  |)^{pq}dx  d\tau
   \right)^\frac1{q} \non\\
 & \le &  C.
 \label{3.17}
\end{eqnarray}
This inequality holds if we assume $p,q\ge 1$ satisfy
\begin{eqnarray}
 pq  &\le& 2,\non\\
 \frac1q + \frac1{q'} &=&  1, \non\\
 \frac{pq'}{2}&\le& 3,\non
\end{eqnarray}
which implies that $p\le \frac32$, and we see that the right-hand side of \eq{3.17} is finite.
Thus the proof of this lemma is complete.

\begin{Lemma} \label{lm3.5} There holds for any $t\in[0,T_{e}]$
\begin{eqnarray}
 \|h_{t}\|_{L^\frac32(Q_{T_e})} \le C.
 \label{3.19}
\end{eqnarray}

\end{Lemma}

\noindent\textbf{Proof.}    From Lemma~\ref{lm3.3} and   equation \eq{3.1}
one can get \eq{3.19} easily.

\subsection{Existence of solutions to the original IBVP}
 \label{sec3.1}

  In this sub-section we shall employ the method of continuation of local
  solutions to prove Theorem~\ref{thm1.1}. To this end, the existence of
  local solutions has been proved in Section~\ref{sec2}, and   {\it a priori}
  estimates established in this section listed mainly in Lemma~\ref{lm3.3}, we then
  assert from these two to conclude the local solutions can be
  extended step by step to a global one.
%

For the uniqueness, we assume there exist two solutions $h^1,\ h^2$ which satisfy \eq{1.7}, and define $h= h^1- h^2$.  Then
$h$ satisfies
\begin{eqnarray}
  && -(h_{t}, \varphi)_{\Omega}  - \frac{\alpha_{1}}{2} (( |h^1_{x}|h^1_{x} + 2B h^1_{x}-
   ( |h^2_{x}|h^2_{x} + 2B h^2_{x}, \varphi_{x} )_{Q_{T_{e}}}  \non\\
  &=&   ( (\alpha_{2}\sigma^1_i + \alpha_{3} )(|h^1_{x}| + B) -(\alpha_{2}\sigma^2_i + \alpha_{3} )(|h^2_{x}| + B) , \varphi) _{Q_{T_{e}}} \non\\
  &=&  (  \alpha_{2}\sigma _i  (|h^1_{x}| + B) + (\alpha_{2}\sigma^2_i + \alpha_{3} )(|h^1_{x}| - |h^2_{x}| ) ,\varphi) _{Q_{T_{e}}},
 \label{3.1unique}
\end{eqnarray}
Choosing test function $\varphi=h$, and noting the monotonicity of the mapping $y\mapsto y|y|$, making use of the H\"older
inequality, we see that \eq{3.1unique} turns out to be
\begin{eqnarray}
   &&\frac12  \|h(t)\|^2 +    \alpha_{1}  B \int_0^t\|h_{x}\|^2  d\tau\non\\
  &\le& \frac12  \|h_0\|^2 +  C\int_{Q_t}  \left(| \sigma_i  (|  h^1_{x}|  + B)h|
  +  | \sigma_i^2|\,  | h^1_{x} - h^2_{x}|\, |h|\right)dxd\tau \non\\
  &\le&   C\int_0^t \left(  \| \sigma_i \|^2 + (\|  (h^1_{x}, \sigma_i^2)\|^2_{L^\infty(\Omega)} + 1\right)\|h\|^2 d\tau
  +  \varepsilon  \int_0^t\| h_{x}\|^2d\tau,
 \label{3.1unique1}
\end{eqnarray}
Taking $\varepsilon= \frac  {\alpha_{1}  B }2$, we infer from \eq{3.1unique1} that
\begin{eqnarray}
  \frac12  \|h(t)\|^2 + \frac  {\alpha_{1}  B }2\int_0^t\|h_{x}\|^2  d\tau
  \le    C\int_0^t \left(  \| \sigma_i \|^2 +  \|  (h^1_{x}, \sigma_i^2)\|^2_{L^\infty(\Omega)} + 1\right)\|h\|^2    d\tau.
 \label{3.1unique2}
\end{eqnarray}
Applying the Gronwall inequality in the integral form one gets
$$
 \|h(t)\|^2=0,
$$
then $h=0$ a.e. in $Q_t$. Thus the uniqueness of solution is proved.

\subsection{Regularity of solutions}
 \label{sec3.2}

To complete the proof of Theorem~\ref{thm1.1}, we still need to investigate the regularity of solutions.
We now assume that $h_0\in H^2_{per}(\Omega)$. Formally differentiating \eq{3.1} with respect to $t$ we arrive at
\begin{eqnarray}
   h_{tt} -   \alpha_{1} \left(( | h_{x}|  +  B)  h_{xx}\right)_t
   + \left((\alpha_{2} \sigma_i  + \alpha_{3} )(|  h_{x}|  + B)\right)_t
   = 0, \label{3.1regul}
\end{eqnarray}

Multiplying \eq{3.1regul}  by $h_t$ and integrating the resulting equation with respect to $x$, recalling
$$
  \left(( | h_{x}|  +  B)  h_{xx}\right)_t = \left(\left(\int^{h_x} (|y| + B) dy\right)_x\right)_t
  = \left(\left(\int^{h_x} (|y| + B) dy\right)_t\right)_x ,
$$
we obtain
\begin{eqnarray}
   & &\frac12\frac{d}{dt}  \| h_{t} \|^2 +  \alpha_{1}  \int_\Omega( | h_{x}|  +  B)  |h_{xt}|^2 dx\non\\
   &=&  - \int_\Omega \left( \alpha_{2} (\sigma_i   )_t (|  h_{x}|  + B)  +  (\alpha_{2} \sigma_i  + \alpha_{3} )
   (|  h_{x}|  )_t\right) h_t dx\non\\
   &\le& C\left( \| ( \sigma_i   )_t \|(\|  h_{x}\|+  1  ) +
   \|  h_{xt}\| ( \|\sigma_i\| + 1  )  \right)\|h_t\|_{L^\infty(\Omega)}  \non\\
   &\le& C   \|  h_{xt}\| \, \|h_t\|_{L^\infty(\Omega)} .
   \label{3.1regul1}
\end{eqnarray}
Applying the Nirenberg inequality in the form $\|f\|_{L^\infty(\Omega)}\le C\|f_x\|^\frac12\|f\|^\frac12 + C'\|f\|$,
we then deduce from \eq{3.1regul1} that
\begin{eqnarray}
   \frac12\frac{d}{dt}  \| h_{t} \|^2 +  \alpha_{1}  \int_\Omega( | h_{x}|  +  B)  |h_{xt}|^2 dx
   &\le& C   (\|  h_{xt}\| +1)  ( \|h_{xt}\|^\frac12 \|h_{t}\|^\frac12+\|h_{t}\|) \non\\
   &\le& C (\|  h_{xt}\|^\frac32 \|h_{t}\|^\frac12 + \|  h_{xt}\|  \|h_{t}\| + \|h_{t}\| +1 ) \non\\
   &\le& \varepsilon \|  h_{xt}\|^2  + C_\varepsilon (\|h_{t}\|^2 + 1)   .
   \label{3.1regul2}
\end{eqnarray}
Here the Young inequality in the form  $ab\le \varepsilon a^p + C\varepsilon b^q$ with various $p,q$ such that
$\frac1p+\frac1q=1$ ($p,q\ge 1$). It is easy to see from the equation \eq{3.1} that $\|h_t(0,x)\|\le C$.
Making use of the Gronwall inequality to \eq{3.1regul2} for $\| h_{t} \|^2$
we then have
\begin{eqnarray}
    \| h_{t} \|^2 + \int_0^t\int_\Omega( | h_{x}|  +  B)  |h_{xt}|^2 dx d\tau    \le  C   .
   \label{3.1regul3}
\end{eqnarray}
Therefore the proof of Theorem~\ref{thm1.1} is finished.

\section{Large-time behavior of the solutions}
 \label{sec4}

This section is devoted to the study of large-time behavior of the solutions, however we need more
assumptions on some parameters, e.g., $\gamma$, $b$. We use $C$ to denote  constants which are
{\it independent} of $t$. In this section we assume that $B>0$.

\subsection{Stationary problem}

The corresponding stationary solution   satisfies
\begin{eqnarray}
  \frac{\alpha_{1}}{2} ( |\hat h_{x}|\hat h_{x} + 2B \hat h_{x},\varphi_{x} )_{\Omega}
  + ((\alpha_{2}\hat \sigma_i + \alpha_{3} )(|\hat h_{x}| + B), \varphi) _{\Omega}
  &=& 0,\label{4.2} \\
  \int_\Omega\hat h(x)dx &=& \alpha,
  \label{4.2a}
\end{eqnarray}
where $\alpha$ is an arbitrarily given constant
 and $\hat \sigma_i = {\rm P.V.}\int_{-\infty}^{\infty}\frac{K\beta \hat h_{x}(x_{1})}{x-x_{1}}dx_{1}$. Noting that
we add the restriction \eq{4.2a}, this similar to the Neumann problem to guarantee the uniqueness of solution and also
makes the existence of solutions easier.
Another difficulty in the proof of the existence and uniqueness of weak solution to problem  \eq{4.2} -- \eq{4.2a}
is due to that there is a non-local term, i.e. $\hat \sigma_i$ in \eq{4.2}. We construct solutions as follows by
replacing $ \sigma_i$.
\begin{eqnarray}
  \frac{\alpha_{1}}{2} ( |\hat h_{x}|\hat h_{x} + 2B \hat h_{x},\varphi_{x} )_{\Omega}
  + ((\alpha_{2}  \sigma_i + \alpha_{3} )(|  h_{x}| + B), \varphi) _{\Omega}
  &=& 0,\label{4.2b} ,
\end{eqnarray}
then it is a usual quasilinear problem and we can establish {\it a priori} estimates as done in the paragraph about estimates
for the stationary solutions, and here omit the details. We then prove the existence.

For the uniqueness, we follow the standard procedure and assume that there are two solutions, denoted by
$h^1(x)$ and $h^2(x)$. Define $h(x) = h^1(x) - h^2(x)$, then $h(x)$ satisfies
\begin{eqnarray}
  \frac{\alpha_{1}}{2} ( | h^1_{x}|  h^1_{x} + 2B   h^1_{x} - (| h^2_{x}| h^2_{x} + 2B h^2_{x}),\varphi_{x} )_{\Omega} &&\non\\
  + ((\alpha_{2} \sigma_i^1 + \alpha_{3} )(| h^1_{x}| + B) - (\alpha_{2} \sigma_i^2 + \alpha_{3} )(| h^2_{x}| + B), \varphi) _{\Omega}
  & = & 0,\label{4.2unique} \\
  \int_\Omega  h(x)dx &=& 0,
  \label{4.2aunique}
\end{eqnarray}
where $\sigma_i^k={\rm P.V.}\int_{-\infty}^{\infty}\frac{K\beta h^k_{x}(x_{1})}{x-x_{1}}dx_{1}$ with $k=1,2$.
Choosing $\varphi=h$, and recalling the monotonicity of  the operator $y\mapsto |y|y$, we arrive at
\begin{eqnarray}
   \alpha_{1}B \| h_{x}\| ^2 +  ( \alpha_{2} (\sigma_i^1 - \sigma_i^2)  (| h^1_{x}| + B) +
    (\alpha_{2} \sigma_i^2 + \alpha_{3} )(| h^1_{x}| - | h^2_{x}|), h) _{\Omega}
  & = & 0,\label{4.2unique1} \\
  \int_\Omega  h(x)dx &=& 0.
  \label{4.2aunique1}
\end{eqnarray}
We need to estimate the second and third terms in the left-hand side of \eq{4.2unique1} as follows.
\begin{eqnarray}
    |  ( \alpha_{2} (\sigma_i^1 - \sigma_i^2)  (| h^1_{x}| + B)  , h) _{\Omega}  |
    &\le& \alpha_{2}\|  \sigma_i^1 - \sigma_i^2\| \,\| (|h^1_{x}| + B)  h\| \non\\
    &\le & C\alpha_{2} \|  h_x\|\, \| |h^1_{x}| + B\|\,\| h\|_{L^\infty(\Omega)} \non\\
    &\le &  C\alpha_{2} \|  h_x\|^2,\label{4.2unique2}
\end{eqnarray}
and
\begin{eqnarray}
 |(\alpha_{2} \sigma_i^2 + \alpha_{3} )(| h^1_{x}| - | h^2_{x}|), h) _{\Omega}  |
 &\le& \| | h^1_{x}| - | h^2_{x}|\|\, \| \alpha_{2} \sigma_i^2 + \alpha_{3}\|\,\| h\|_{L^\infty(\Omega)} \non\\
  &  \le &  (C\alpha_{2} \|h^2_x\| + \alpha_{3})\|   h _{x} \|^2,
  \label{4.2unique3}
\end{eqnarray}
then   from  \eq{4.2unique2} and \eq{4.2unique3} we find \eq{4.2unique1}  turns out to be
\begin{eqnarray}
   \alpha_{1}B \| h_{x}\| ^2
  & \le & (C\alpha_{2} \|h^2_x\| + \alpha_{3})\|   h _{x} \|^2 . \label{4.2unique4}
\end{eqnarray}
Therefore if $\alpha_{1}$ is suitably large, then one gets
\begin{eqnarray}
     \| h_{x}\| ^2  \le  0,\label{4.2unique5}
\end{eqnarray}
it implies that $h_{x}$ is a.e. equal to $0$, i.e.  $h $ is a.e. equal to a constant $c$,
recalling the condition \eq{4.2aunique1}, and conclude that $h=0 $ a.e. in $\Omega$.

\vskip0.2cm
We are now going to derive some estimates on $\hat h$.

\vskip0.5cm
\noindent{\it Estimates of the stationary solution.} First,
choosing test function $\varphi = \hat h$ we obtain
\begin{eqnarray}
   \frac{\alpha_{1}}{2} ( |\hat h_{x}|\hat h_{x} + 2B \hat h_{x},\hat h_{x} )_{\Omega}
  + ((\alpha_{2}\hat \sigma_i + \alpha_{3} )(|\hat h_{x}| + B)  , \hat h) _{\Omega}   =0.
 \label{4.3}
\end{eqnarray}
From which we infer that
\begin{eqnarray}
   \int_\Omega \frac{\alpha_{1}}{2}(|\hat h_{x}| ^3   + 2B |\hat h_{x}|^2 + \alpha_{2}B\hat \sigma_i\hat h)dx
  &=&  \int_\Omega (\alpha_{2}\hat \sigma_i   + \alpha_{3} ) |\hat h_{x}|\hat h +\alpha_{3}B \hat h) dx\non\\
   & \le & \alpha_{2} \int_\Omega (( C +  \varepsilon  )|\hat h_{x}| ^3+ C_\varepsilon) dx + \alpha_{3}B\alpha .\ \
 \label{4.4}
\end{eqnarray}
Hence if $\alpha_{1}$ is sufficiently large,  for instance we can take $\gamma$   large, we then have
\begin{eqnarray}
   \|\hat h_{x}\|_{L^3(\Omega)},\  \|\hat h_{x}\|_{L^2(\Omega)}\le  \delta_1.
 \label{4.5}
\end{eqnarray}
Here $\delta_1$ is a suitably  small constant.

Integrating \eq{4.3} by parts one gets
\begin{eqnarray}
   \left(  (-\alpha_{1} \hat h_{xx} + \left(\alpha_{2}\hat \sigma_i + \alpha_{3}  )\right)
   (|\hat h_{x}| + B), \varphi \right) _{\Omega}   =0,
 \label{4.6}
\end{eqnarray}
 Choosing $\varphi = \hat h_{xx}$ we obtain easily
\begin{eqnarray}
   \| \hat h_{xx}\| +  \| \hat h_{x}\|_{L^\infty(\Omega)} \le  \delta_2,
 \label{4.5a}
\end{eqnarray}
where $\delta_2$ is small if $\alpha_{1}$ is sufficiently large. In the following context we
set $\delta:=\max\{\delta_1,\delta_2\}$, then \eq{4.5} and \eq{4.5a} hold simultaneously.

\subsection{Asymptotic behavior}

To investigate the large-time behavior of solution, we first reformulate a
problem for the difference of solution $h$ and the stationary solution $\hat h$,
which is more convenient to handle.

\noindent{\it Reformulation of the problem.} We define
$$
 u(t,x) = h(t,x) - \hat h(x),
$$
then it is easy that $u$ satisfies for almost every $(t,x)\in Q_{T_{e}}$, the problem
\begin{eqnarray}
 ( u_{t},\varphi)_\Omega - ((\alpha_{1} u_{xx} -  \alpha_{2}\tilde{\sigma}_i  )(| h_{x}| + B) ,
 \varphi)_{Q_{T_e}} = ({\cal G }, \varphi)_{Q_{T_e}}   .
 \label{4.7}
\end{eqnarray}
for all $\varphi\in C^\infty_{per}(\Omega)$. Here $\tilde{\sigma}_i :=  {\sigma}_i - \hat \sigma_i$
 and
$$
 {\cal G } :=  (\alpha_{1} u_{xx}  -  \alpha_{2}\tilde{\sigma}_i  )(| h_{x}| - | u_{x}|) .
$$
It is easy to get
\begin{eqnarray}
 |{\cal G } | &\le&  |  \alpha_{1} u_{xx}  -  \alpha_{2}\tilde{\sigma}_i |\, | h_{x}  -   u_{x}|
 = |  \alpha_{1} u_{xx}  -  \alpha_{2}\tilde{\sigma}_i |\, | \hat h_{x}   | \non\\
  &\le&  \delta |  \alpha_{1} u_{xx}  -  \alpha_{2}\tilde{\sigma}_i | ,
 \label{4.7a}
\end{eqnarray}
where \eq{4.5a} was used. The boundary and initial conditions now turn out to be
\begin{eqnarray}
 u(t,c) &=&  u(t,d),\quad   u_{x } (t,c) =  u_{x} (t,d),
 \label{4.7b}\\
 u(0,x) &=&  h_{0}(x)  -    \hat h_{x} .
 \label{4.7c}
\end{eqnarray}

We are going to prove there exists a global solution to problem \eq{4.7} -- \eq{4.7c} and investigate
its large-time behavior under suitable assumptions.
\begin{Proposition} Assume that $\| h_{0}  -   \hat h \|_{H^1(\Omega)}
\le \varepsilon$ and that there  exists a sufficiently large
$N$ such that $\gamma,\alpha_1\ge N$. Then there exists a unique global solution
$u=u(t,x)$ to  problem \eq{4.7} -- \eq{4.7c} such that $u\in L^\infty(0,\infty;H^1_{\rm per}(\Omega)) ,
u_x\in L^\infty(0,\infty;H^1_{\rm per}(\Omega))$.
Moreover    the solution $u$ converges, as $t\to\infty$, to the stationary
solution in the following sense
\begin{eqnarray}
 \sup_{c\le x\le d}|u(t,x)-\bar u(t,x)  | \to 0.
 \label{4.7e}
\end{eqnarray}

\end{Proposition}

In the following context, we assume that
\begin{eqnarray}
 \sup_{0\le t\le T_e}\|u(t)\|_{H^1(\Omega)}\le M < 1.
 \label{4.7d}
\end{eqnarray}

\noindent{\bf Proof.}
For the sake of convenience we derive from weak form of \eq{4.7}  we will
use the following form
\begin{equation}
 u_{t} - (\alpha_{1} u_{xx} -   \alpha_{2}\tilde\sigma_i  )(|h_{x}| + B)  = {\cal G}
 \label{4.1ae}
\end{equation}
for almost all $(t,x)\in Q_{T_{e}}$.
This  form is   more convenient in the investigation of large-time behavior
of the solution $h$.

Multiplying equation \eq{4.1ae}  by $\alpha_{1} u_{xx}  - \alpha_{2}\tilde{\sigma}_i  $ and integrating the
resulting equation with respect to $x$ over $\Omega$, we obtain
\begin{eqnarray}
  \frac{\alpha_{1}}2\frac{d}{dt} \|u_{x}\|^2   &+& \alpha_{2}\int_\Omega  \tilde{\sigma}_i u_t  dx
   +  \int_\Omega(\alpha_{1} u_{xx}  -  \alpha_{2}\tilde{\sigma}_i  )^2(| h_{x}| + B)dx\non\\
   & = & \int_\Omega {\cal G }  (  \alpha_{1} u_{xx}  -  \alpha_{2}\tilde{\sigma}_i )  dx \label{4.8a}\\
   & \le & \delta \int_\Omega   |  \alpha_{1} u_{xx}  -  \alpha_{2}\tilde{\sigma}_i |^2  dx.
 \label{4.8}
\end{eqnarray}
In a similar manner as in section of {\it a priori} estimates, we find
that $\int_\Omega  \tilde{\sigma}_i u_t  dx$
can be rewritten as a time derivative of a non-negative function $G(t)$. Thus
\begin{eqnarray}
   && \frac12\|u_{x}\|^2    + \alpha_{2}G(t)
   +  \int_0^t\int_\Omega\left(\alpha_{1} u_{xx}  - \alpha_{2}\tilde{\sigma}_i  \right)^2
   (| h_{x}| + B)dxd\tau \non\\
   & \le & \frac12\|u_{0x}\|^2   + \alpha_{2}G(0) +
   \delta \int_0^t\int_\Omega   |  \alpha_{1} u_{xx}  -  \alpha_{2}\tilde{\sigma}_i |^2  dxd\tau.
 \label{4.9}
\end{eqnarray}
 Letting $\delta \le \frac{B}{2}$, we obtain form  \eq{4.9} that
\begin{eqnarray}
 C &\ge& \int_0^t\int_\Omega\left(\alpha_{1} u_{xx}  - \alpha_{2}\tilde{\sigma}_i  \right)^2 dxd\tau\non\\
   & = & \int_0^t\int_\Omega\left(\alpha_{1}^2| u_{xx} |^2 - 2 \alpha_{1}   \alpha_{2} u_{xx}\tilde{\sigma}_i
    + \alpha_{2}^2|\tilde{\sigma}_i |^2 \right) dxd\tau,
 \label{4.10}
\end{eqnarray}
it is easy to show formally (to justify this, we can, for instance,  construct smooth solutions) that
\begin{eqnarray}
 - \int_0^t\int_\Omega    u_{xx}\tilde{\sigma}_i   dxd\tau
 &=&  \int_0^t \sum_{n=1}^{\infty} n^3(a_n^2 + a_{-n}^2)d\tau\non\\
 &\ge& 0,
 \label{4.11}
\end{eqnarray}
hence we infer from \eq{4.10} that
\begin{eqnarray}
  \int_0^t\int_\Omega (| u_{xx} |^2 + |\tilde{\sigma}_i |^2) dxd\tau
  \le C .
 \label{4.12}
\end{eqnarray}
Invoking that the periodic boundary conditions for $u$, we further get that there
exists a point say $x_0(t)\in (c,d)$
such that  $u(t,x_0(t))=0$, whence the Poincar\'e inequality of the form
$\|f\|\le C \|f_x\|$ for all $f\in H^1(\Omega)$, and
it follows from \eq{4.12} that
\begin{eqnarray}
  \int_0^t \| u_{x} \|^2 d\tau \le C.
 \label{4.13}
\end{eqnarray}

Moreover formally
we multiply \eq{4.1ae} by $u_t(| h_{x}| + B)^{-\frac12}$, make use of the H\"older inequality,
estimates \eq{4.9} and \eq{4.7a}, and the fact that $| h_{x}| + B\ge B$  to obtain
\begin{eqnarray}
 && \int_\Omega  \frac{ |u_{t}|^2}{| h_{x}| + B}  dx  + \frac{d}{dt}\left(\frac{\alpha_{1}}{2} \|u_{x}\|^2 + \alpha_{2}G(t) \right) \non\\
   &=& \int_\Omega {\cal G}  \frac{ |u_{t}| }{| h_{x}| + B}  dx =  \int_\Omega\frac{{\cal G}   }{(| h_{x}| + B)^\frac12}
    \frac{ |u_{t}| }{(| h_{x}| + B)^\frac12}  dx\non\\
   &\le& C\left(\int_\Omega  \frac{ |\alpha_{1}u_{xx} - \alpha_{2} \tilde\sigma_i|^2  }{| h_{x}| + B} dx\right)^\frac12
     \left(\int_\Omega\frac{ |u_{t}|^2 }{| h_{x}| + B}  dx\right)^\frac12   \non\\
   &\le& C\left(\int_\Omega    |\alpha_{1}u_{xx} - \alpha_{2} \tilde\sigma_i|^2  B^{-1} dx\right)^\frac12
     \left(\int_\Omega\frac{ |u_{t}|^2 }{| h_{x}| + B}  dx\right)^\frac12  ,
 \label{4.13a}
\end{eqnarray}
Integrating \eq{4.13a} with respect to $t$ yields
\begin{eqnarray}
  \int_0^t\int_\Omega  \frac{ |u_{t}|^2}{| h_{x}| + B}     dxd\tau     \le C  .
 \label{4.13b}
\end{eqnarray}

Next with the help of \eq{4.13b} we can get more estimate from   \eq{4.8a}, and write
\begin{eqnarray}
  \int_0^t\int_\Omega  \tilde{\sigma}_i u_t  dxd\tau
     = \int_0^t\int_\Omega  \tilde{\sigma}_i  (|h_x|+B)^\frac12 \frac{u_t}{(|h_x|+B)^\frac12}
   dxd\tau ,
 \label{4.14a}
\end{eqnarray}
then applying the H\"older inequality and the Sobolev embedding theorem, we estimate
\begin{eqnarray}
 \left|\int_0^t\int_\Omega  \tilde{\sigma}_i u_t  dxd\tau\right|
   & \le &  C  \left(\int_0^t(\|h_x\|_{L^\infty(\Omega)}+1)\| \tilde{\sigma}_i\|^2d\tau \right)^\frac12
   \left(\int_0^t\int_\Omega \frac{|u_t|^2}{ |h_x|+B }  dxd\tau\right)^\frac12\non\\
   & \le & C  \left(\int_0^t(\|u_x\|_{L^\infty(\Omega)}+1)\| u_x\|^2)d\tau \right)^\frac12 \non\\
   & \le & C  \left(\int_0^t(\|u_x\|_{L^\infty(\Omega)}^2 + \| u_x\|^4+\| u_x\|^2) d\tau \right)^\frac12\non\\
   & \le & C  \left(\int_0^t \|u_x\|_{H^1(\Omega)}^2   d\tau \right)^\frac12\non\\
   & \le & C  .
 \label{4.14}
\end{eqnarray}
We used here the estimates \eq{4.12} and \eq{4.13a}. Thus applying
  \eq{4.8} we infer from \eq{4.8a} that
\begin{eqnarray}
  \int_0^t \left|\frac{d}{dt} \|u_{x}\|^2\right|  d \tau
   &+& \int_0^t \int_\Omega(\alpha_{1} u_{xx}  -  \alpha_{2}\tilde{\sigma}_i  )^2(| h_{x}| + B-\delta )dx\non\\
   & \le &  \int_0^t  \left|\alpha_{2}\int_\Omega  \tilde{\sigma}_i u_t  dx \right|  d \tau  \non\\
   & \le &    C.
 \label{4.15a}
\end{eqnarray}
If we choose $ \delta$ is suitably small such that $B\ge \delta$, then the following estimate is obtained
\begin{eqnarray}
  \int_0^t \left|\frac{d}{dt} \|u_{x}\|^2\right|  d \tau       \le      C.
 \label{4.15}
\end{eqnarray}

To investigate the large-time behavior of the solution, we shall apply the following simple lemma
\begin{Lemma} Suppose that $f=f(t)$ is a non-negative function defined over $[0,\infty)$,
and is differentiable, such that
$$
 \int_{0}^{\infty} \left(f(t) + |f'(t)|\right)dt<\infty,
$$
then $f(t)\to 0$, as $t \to \infty$.
\end{Lemma}

Now we define $ f(t):=\|u_{x}(t)\|^2$, then from \eq{4.13} and \eq{4.15} it follows that
$$
 \int_0^\infty\left( f(t) + \left|\frac{d}{dt}f(t)\right|\right)dt \le C,
$$
and conclude that
$$
 \|u_{x}(t)\|^2\to 0 , \ {\rm as}\  t\to\infty.
$$
Hence making use of the Poincar\'e inequality again  we obtain
$$
  \|u  - \bar u \|_{L^\infty(\Omega)}(t)  \to 0.
$$
Namely as $t\to\infty$,
$$
  \sup_{c\le x\le d} \left|h(t,x)  - \hat h(x)
  - \frac{1}{|\Omega|}\left(\int_\Omega  h(t,x)dx  -  \alpha \right) \right|  \to 0.
$$
Therefore we complete the proof of  Theorem~\ref{thm1.3}.

\section{Comparison with numerical solutions}
 \label{sec5}

We will employ the spectral method based on the Fast Fourier Transform
in this section to perform numerical simulation about the results in
Theorem~\ref{thm1.3}, especially the convergence of solution $h$ to
 its steady state $\hat h$.
We successfully reproduce the numerical benchmark case from \cite{ZHXS17} for the following problem
\begin{equation}
h_t = M_d\left((\sigma_i + \tau)b + \Psi H - \gamma h_{xx}H\right) \left( |h_x| + B \right),
\label{5.1}
\end{equation}
with initial and boundary conditions \eq{1.2} -- \eq{1.3}.
The equilibrium state of the system is determined solely by the applied stress $\tau$ and is independent
of the dissipation parameter $B$. Specifically, the equilibrium profile satisfies
in the sense of Definition~\ref{def1.3} the following  equation
\begin{equation}
 (\sigma_i + \tau) b + \Psi H - \gamma h_{xx} H = 0,
 \label{5.2}
\end{equation}
with periodic boundary conditions.
\begin{figure}[htbp]
 \begin{center}
	\includegraphics[width=5.8in]{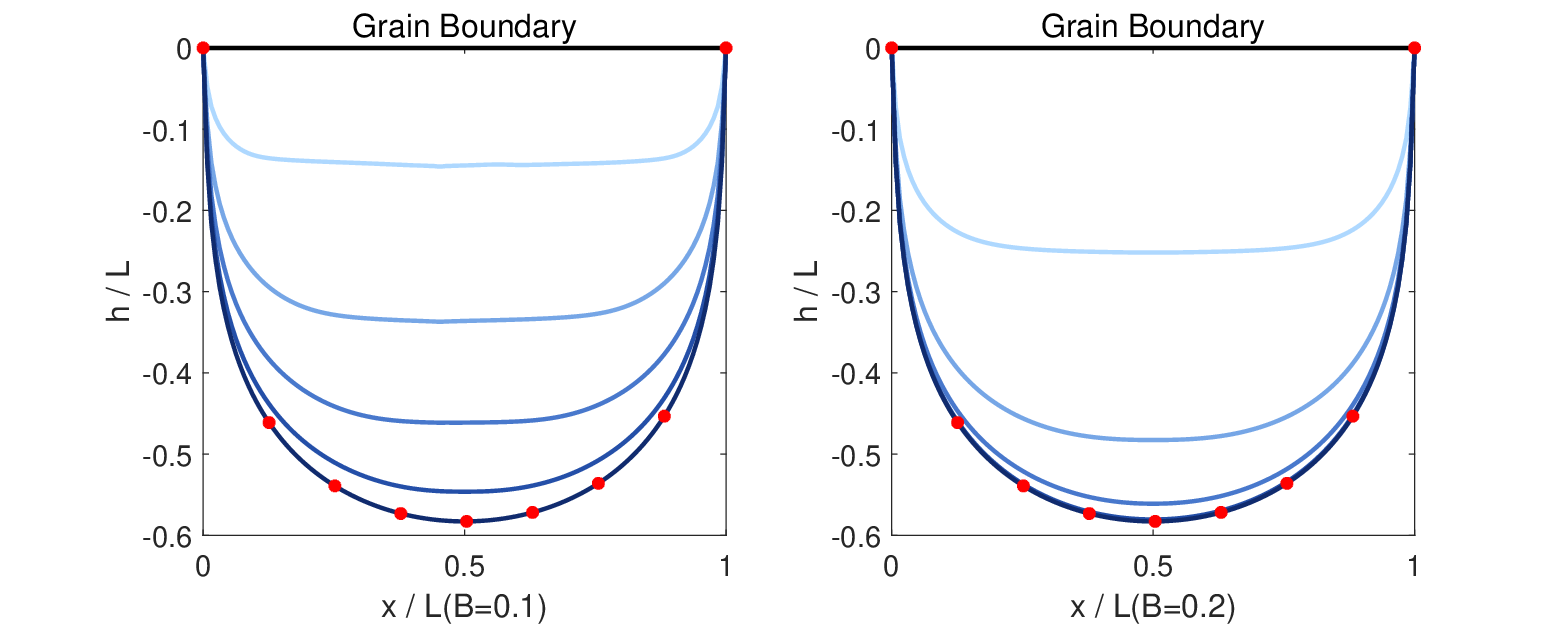}
 \end{center}
 \caption{ Comparative analysis of grain boundary (GB) evolution dynamics
   under $\tau = 0.1\mu$ shear stress with bilateral pinning. The lines
   represent GB profiles at characteristic time steps:
		$t\in \{0,2t_0,5t_0,8t_0,12t_0,+\infty\}$ where $t_0=L/(M_d\gamma)$.
  The color deepens with increasing time.  Left figure: $B = 0.1$; right
  figure: $B = 0.2$. The red stars (\textcolor{red}{$\star$}) represent
  the analytical equilibrium solution given by Eq.~\eqref{5.2}, toward
  which the simulated GB profiles converge over time. }
  \label{fig:growth}
\end{figure}

We adopt the same physical parameters as in Ref.~\cite{ZHXS17}.
The simulations are conducted for aluminum, with shear modulus $\mu = 26.5$~GPa,
Poisson's ratio $\nu = 0.347$, and lattice constant $a_0 = 4.0495$~\AA. The grain
boundary parameters include a Burgers vector magnitude $|\mathbf{b}| = a_0/\sqrt{10}
\approx 1.28$~\AA, step height $H = a_0/\sqrt{10} \approx 1.28$~\AA, and energy
density $\gamma = 0.5$~J/m$^2$. The simulation domain has a length of $L = 10^{-7}$~m,
and the disconnection mobility is set to $M_d = 1.0$~m$^2$/J$\cdot$s. To efficiently
compute the nonlocal stress integral $\sigma_i(x,t)$ and spatial derivatives $h_{x}$
and $h_{xx}$, we implement a spectral method based on the Fast Fourier Transform (FFT).

As shown in Fig.~\ref{fig:growth}, under applied shear stress $\tau = 0.1\mu$, the
grain boundary undergoes a morphological transition from an initially flat interface
to a stable curved configuration. The evolution process exhibits a clear dependence
on the dissipation parameter $B$: higher values of $B$ lead to faster relaxation
while maintaining the same final equilibrium profile. The equilibrium shapes
$h_{\mathrm{eq}}(x)$ predicted by the continuum model~\eq{5.1},
show excellent agreement with the analytical solution of Eq.~\eqref{5.2}.

\vskip0.5cm
\noindent {\bf Acknowledgements.}
  The first author of this article is supported in part by Science and
  Technology Commission of Shanghai Municipality (Grant No. 20JC1413600)
  and by Tongji University Medicine-X Interdisciplinary Research
   Initiative under contract No. 2025-0553-ZD-11;
   The second author is supported in part by National Natural Science Foundation of China 12401569 and Shanghai Sailing Program 24YF2712700;
    The  third author is supported in part  by the Hong Kong Research
  Grants Council General Research Fund 16309825.


\end{document}